\documentclass{amsart}

\usepackage{amsmath,amsthm,amssymb,tikz-cd,todonotes,mathrsfs,stmaryrd}
\usetikzlibrary{decorations.pathreplacing,angles,quotes}

\theoremstyle{plain}
\newtheorem{thm}{Theorem}

\newtheorem*{theorem*}{Theorem}

\theoremstyle{definition}

\theoremstyle{remark}

\newcommand{\nc}{\newcommand}
\nc{\dmo}{\DeclareMathOperator}

\nc{\B}[1]{\mathbb{#1}}
\nc{\C}[1]{\mathcal{#1}}
\nc{\Sc}[1]{\mathscr{#1}}
\dmo{\sbld}{\preceq_s}
\dmo{\bld}{\preceq}
\dmo{\bldneq}{\prec}
\dmo{\sbldneq}{\prec_s}
\nc{\lb}{\llbracket}
\nc{\rb}{\rrbracket}
\dmo{\Id}{Id}

\renewcommand{\mod}[1]{{\ifmmode\text{\rm\ (mod~$#1$)}\else\discretionary{}{}{\hbox{ }}\rm(mod~$#1$)\fi}}

\begin{document}
\title[] {No  function is continuous only at points in a  countable dense subset}

\author[Silva]{Cesar E. Silva}
\address[Cesar E. Silva]{Department of Mathematics\\
     Williams College \\ Williamstown, MA 01267, USA}
\email{csilva@williams.edu}

\author[Wu]{Yuxin Wu}
\address[Yuxin Wu ]{ \\
     Williams College \\ Williamstown, MA 01267, USA}
\email{yw2@williams.edu}


\maketitle



\begin{abstract}  We give a short  proof, that can be used in an introductory real analysis course,  that if a  function that is defined on the set of real numbers 
  is continuous on  a  countable dense set, then it is continuous on  an uncountable set.  This is done for functions defined  on   complete metric spaces without isolated points,  and the argument only uses that Cauchy sequences  converge, and we prove the version related to Volterra's theorem.  We discuss how this theorem is a direct consequence of the Baire category theorem, and also discuss  Volterra's theorem and the history of this problem. We conclude with a simple example, for each complete metric space without isolated points and each set that is a countable union of closed subsets,  of a real-valued function that is discontinuous only on that set.
 	
 	 \end{abstract}


\section{Introduction.} A function defined on the set of real numbers that is continuous only at 0 is given by $f(x)=xD(x)$, where 
$D(x)$ is Dirichlet's function (i.e., the indicator function of the set of rational numbers).
From this one can construct examples of functions continuous at only finitely many points, or only  at all the  integer points.
It is also possible to construct a function  that is continuous only on the set of irrational numbers; a well-known  example is Thomae's function, also called the generalized Dirichlet function---and see also the example in Section~\ref{examples}.  The question arises whether there is a function defined on the real numbers that is continuous only on the set of rational numbers (i.e., continuous at each rational number and discontinuous at each irrational), and the answer has long been known to be no. Dunham discusses in \cite{Du} Volterra's 1881 theorem that implies this result. The standard proof now proceeds by arguing that the set of continuity points of a function must be a countable intersection of open sets, and that by Baire's category theorem the set of rational numbers $\mathbb Q$ is not a countable intersection of open sets. As Dunham \cite{Du}  points out, Volterra's theorem appeared before Baire's theorem, and in \cite{Du} he gives
a nice short argument of Volterra's theorem.

We prove that for complete metric spaces, without isolated points,  there are no  functions  that are continuous  only on a countable dense set. This, of course, also follows from Baire's theorem by the previous remarks. Our argument, though,  only uses the fact that Cauchy sequences must converge.  As in \cite{Du}, we are interested in techniques independent of Baire's theorem that can be used in an undergraduate real analysis class. The main idea is that for each sequence of points one can  construct a new
point, not in the sequence,  that is obtained in the intersection of a nested decreasing sequence of closed balls. This idea of creating a new point in a 
countable intersection of closed bounded intervals can already be seen in Cantor's 1874 proof of the uncountability of the real numbers (see, e.g., \cite{KS}), and Volterra's original proof does use a version of the Nested Intervals theorem as remarked in \cite{Du}. There are proofs already that use the Nested Intervals theorem to prove that there is no   function that is continuous only on the rational numbers, or on a countable dense set of real numbers, see e.g., Gauld \cite{G93}, Lee \cite{L}, and Saxe \cite{Sa}.  We   show that with a slight modification, whose ideas are already contained in Volterra's \cite{Vo} paper, our proof can be used to obtain Volterra's result for complete metric spaces, Theorem~\ref{T2}; in fact, we obtain a slightly stronger result than Volterra's as stated by Gauld \cite{G93}, which treats the case for functions defined on an interval (Gauld mentions that  his proof also works for compact metric spaces and our remark is that it works in the setting of complete metric spaces).

 \medskip

\section {Historical Remarks.}
We now make a few remarks regarding the history of this problem.  First we discuss the  theorems and then the examples.
The story of the theorems begins with Volterra's theorem of 1881. Let  $C(f)$ stand for the set of points where the  function $f$ is continuous,
and $D(f)$ for its set  of discontinuity points, or its complement.  Volterra proved that if $f$ and $g$ are real-valued functions defined on 
$\mathbb R$ and such that $C(f)$ and $C(g)$ are dense, then $C(f)\cap C(g)$ is nonempty (in fact, this set is dense as noted in \cite{G93}). From this it follows that there is no function 
$f$ such that $C(f)=\mathbb Q$, since in Volterra's theorem one can choose $g$ to be Thomae's function, for which $C(g)=\mathbb Q^c$. To see Volterra's theorem  as a
consequence of Baire's category theorem, which was published eight years later in 1899,  we recall some  facts about sets. If a set $A$ is a countable union of closed sets, i.e., $A=\bigcup_{n=1}^\infty F_n$ where each $F_n$ is closed (such a set is called an $\mathcal F_\sigma$ set), 
then its complement $A^c=\bigcap_{n=1}^\infty F_n^c$ is a countable intersection of open sets (a set of this form is called a $\mathcal G_\delta$ set), and vice versa. 
It has been known for some time  that $D(f)$ is an $\mathcal F_\sigma$ set. There is a proof of this  for functions of a real variable in 
 Lebesgue \cite[p. 235]{Leb}, and the reader may find a nice short proof in Oxtoby \cite[Theorem 7.1]{Ox80}. (We learned of the reference to Lebesgue in Renfro \cite{Ren}, where one can also find references to many other sources).  This  has been generalized: as Bolstein \cite{Bo73} points out, there is an outline of a proof for functions on topological spaces (which are more general than metric spaces) in Hewit-Stromberg \cite[Exercise 6.90] {HS65}.  It follows then 
  that  the set $C(f)=D(f)^c$ is a countable intersection of open sets. 
 
  Now, Baire's category theorem states that  
 in a complete metric space $X$, a countable intersection of dense open sets is dense, and equivalently, 
that $X$  is not a countable union of nowhere dense sets (a nowhere dense set is a set whose closure has empty interior); see e.g., \cite[3.3.2, 3.3.3]{Si08}.  It follows then that if $C(f)$ and $C(g)$ are dense, then their intersection is dense, so nonempty. 
 In fact, we obtain more: that if we have a countable family of functions $f_i$  so that each $C(f_i)$ is dense, then the intersection of the family is dense. 
Gauld and Piotrowski  \cite{GP93} defined strongly Volterra spaces,  which  were later called Volterra spaces, as those for which  $C(f)\cap C(g)$ is dense when $C(f)$ and $C(g)$ are dense. Clearly spaces that satisfy  Baire's theorem, such as complete metric spaces, are Volterra spaces, and the fact that the converse is not true was shown in \cite{GP96, GL00}---spaces for which Volterra's theorem holds have been studied in e.g.,  \cite{CaGa05, GL00}.
 Finally we note that Gauld gave a proof in  \cite{G93}, just using the Nested Intervals theorem in $\mathbb R$, that $C(f)\cap C(g)$ is in fact uncountable.
 The fact that this intersection is uncountable  also follows from Baire's  theorem since it were countable then $X$ would be a countable union of 
 nowhere dense sets. Gauld also mentions that his methods generalize to compact metric spaces. As in  
  \cite{G93}, we are interested in giving a self-contained proof of these results, and instead of using the Nested Intervals theorem, 
  we use Cauchy completeness.

 Now we consider the examples. A natural question  to ask now is  if a set that is a countable union of closed sets admits a function with that set as its discontinuity points; for the case of the rationals the answer is yes and given by  Thomae's function. For the general case on the real line this was shown in 1903   by W. H Young  \cite{Yo03} and in 1904  by Lebesque \cite{Leb}; there is a simple construction in Oxtoby \cite[7.2]{Ox80}. The construction was generalized to  metric spaces  by H. Hahn in 1932 \cite{Ha32} (we learned of the paper  of  Hahn in Bolstein \cite{Bo73}). Then  Bolstein \cite{Bo73} proved this fact for a large class of topological spaces that includes, for example, locally compact Hausdorff spaces.  More recently, Kim \cite{K99} constructed  a function that is continuous on any countable intersection of open sets in an arbitrary metric space (without isolated points), although does not seem aware of the earlier results of Young-Lebesgue-Hahn, and Bolstein. Our example in Section~4 is somewhat simpler, other constructions are in \cite{GP93, GP96}.

\section {The Theorems.}

We assume our metric spaces are nonempty. Let $B(p,r)$ and $ B[p,r]$ denote open and closed balls centered at $p$ and of radius $r$, respectively. We recall that a function $f:X\to \mathbb R$ is said to be continuous on a set $A\subset X$ if $f$ is continuous at each point of $A$. This is different from the restriction of $f$ to the set $A$ being continuous (as is illustrated in the last paragraph of this section.

\begin{thm}\label{T1}
Let $X$ be a complete metric space without isolated points. 
If a real-valued function on $X$ is continuous on   a  dense set, then it is continuous on   an uncountable dense set.
\end{thm}

\begin{proof} Let $f:X\to\mathbb R$ be a function  and let $A\subset X$ be a countable dense set such that $f$ is 
continuous at every point in $A$. We will first show that there is a point $z\in X\setminus A$ such that $f$ is continuous at $z$.
Write $A=\{q_n:n\in\mathbb N\}$. We define, by induction, a nested sequence 
 of closed balls  $(B[p_n, r_n])_{n\in\mathbb N}$ such that for each $n\in\mathbb N$, 
 \begin{align*}
 r_n&<\frac{1}{2^n},
 \ q_n\notin B[p_n, r_n],\text{ and }\\
|f(x)-f(y)| &<  \frac{1}{2^{n-1}}\ \text{ for all }x,y\in B[p_n, r_n] .
\end{align*}

We start by   setting $p_1=q_1$. By continuity at $p_1$, 
there exists an open ball $B(p_1,\delta_1)$ such that 
$|f(x)-f(p_1)|<\frac{1}{2}$ for all $x$ in $B(p_1,\delta_1)$.  Choose $C_1$ 
a closed ball with center in $X$,  of radius $r_1$ with  $0<r_1<1/2$,  contained in $B(p_1,\delta_1)$  and missing $p_1$
(which we can do as $p_1$ is not an isolated point). 

For the inductive step suppose  that we are given an open ball $B(p_{n},\delta_{n})$, where $p_{n}\in A$, 
 such that \[|f(x)-f(p_{n})|<\frac{1}{2^{n}}\text{ for all }x \in B(p_{n},\delta_{n}),\]  
and a closed ball $C_{n}$,  of radius $r_{n}$, inside $B(p_{n},\delta_{n})$ and not containing  
$q_{n}$. Choose now  an element of $A$, denoted by $p_{n+1}$, in $B(p_{n},\delta_{n})$.
By continuity at $p_{n+1}$, there exists an open ball $B(p_{n+1},\delta_{n+1})$, 
that we can choose inside $B(p_{n},\delta_{n})$, such that 
$|f(x)-f(p_{n+1})|<\frac{1}{2^{n+1}}$ { for all }$x \in B(p_{n+1},\delta_{n+1}).$ Choose a closed 
ball $C_{n+1}$ inside $B(p_{n+1},\delta_{n+1})$, of radius $r_{n+1}<\frac{r_{n}}{2}$, and that misses the element $q_{n+1}$. 
For the final condition, if $x,y\in C_n$, then 
\[|f(x)-f(y)|\leq |f(x)-f(p_n)+|f(p_n)-f(y)| <\frac{1}{2^n}+\frac{1}{2^n}=\frac{1}{2^{n-1}}.\]

We have constructed a sequence of nested closed balls $C_n$ with $q_n\notin C_n$. The centers of these balls form a Cauchy sequence, as their radii $r_n$ satisfy $s_n<1/2^n$ for $n\geq 1$. Since the space is complete, the sequence of centers has a limit $z$, which must be in the intersection of the balls. 
Also, $z$ is not in  $A$ as $z\neq q_n$ for all $n\in\mathbb N$. 

We show that  $f$ is continuous at $z$.
Let $\varepsilon>0$ and choose $n\in\mathbb N$ so that
$1/2^{n-1}<\varepsilon$. Choose $\delta>0$ such that the ball $B(z,\delta)$ is in $C_n$.
If $x\in B(z,\delta)$, then 
$|f(x)-f(z)|< \frac{1}{2^{n-1}}<\varepsilon.$

Now let $B$ be the subset of $X\setminus A$ consisting of all points $z$ such that $f$ is continuous at $z$.
If $B$ is countable let $A^\prime=A\cup B$. Then $A^\prime$ is countable and dense, so by the previous argument,
there exists a point $z$ in $X\setminus A^\prime$ such that $f$ is continuous at $z$. But this contradicts 
the definition of $B$, therefore $B$ must be uncountable.
\end{proof}

If  $X$ consists only of  isolated points and $A$ is dense, then $A=X$. The theorem  also implies that a 
(nonempty) complete metric space with no isolated points must be 
uncountable (consider a   function constant on $X\setminus \{p\}$ and with a different value at $p$).  It follows that positive-radius closed balls in $X$, as they are complete, must be uncountable. Therefore, nonempty open balls in $X$ are uncountable.

We now show that a small modification of the proof in Theorem~\ref{T1} obtains Volterra's theorem \cite{Vo} as strengthened by Gauld \cite{G93}.

\begin{thm}\label{T2} Let $X$ be a nonempty complete metric space without isolated points and let $f$ and $g$ be two real-valued functions defined on $X$. If $C(f)$ and $C(g)$ are dense, then $C(f)\cap C(g)$ is an uncountable dense set.
\end{thm}

\begin{proof}    Let $G$ be an open set in $X$.
Since $C(f)$ is dense we can choose $q_{1}\in C(f)$ and $\delta_1>0$ so that    $B(q_1,\delta_1)\subset G$  and 
$|f(x)-f(q_{1})|< \frac12$ for all $x\in B(q_1,\delta_1)$.  Since $C(g)$ is dense we can choose $r_{1}\in C(g)\cap B(q_1,\delta_1)$ and $\delta_1^\prime>0$ so that $B(r_{1},\delta_1^\prime)\subset B(q_1,\delta_1)$ and 
$|g(x)-g(r_{1})|<\frac{1}{2}$ for all $x$ in $B(r_{1},\delta_1^\prime)$.  Now choose $C_1$ 
a closed ball, of radius $1/2>s_1>0$,  contained in $B(r_{1},\delta_1^\prime)$.

The inductive step is similar. We can assume that by induction we have constructed a sequence $(C_n)$ of closed balls, in $G$,  and $p_n\in C(f), q_n\in C(g)$   such that  
\[|f(x)-f(p_{n})| <\frac{1}{2^{n}}\text{ and } |g(x)-g(q_{n})| <\frac{1}{2^{n}}\text{ for all }x\in C_n.\]
There is a point $z$ in the intersection $\bigcap_n C_n$ 
so that the continuity condition holds for both $f$ and $g$ at $z$.  The argument for this is similar to the one in
Theorem~\ref{T1}.  So $z\in C(f)\cap C(g)$ and  it is clear that $z$ is  also  in $G$. 

Now if $C(f)\cap C(g)$  were countable, say equal to the set $\{t_n:n\in\mathbb N\}$, we could repeat the construction choosing
$C_n$ to avoid $t_1,\ldots,t_n$. Then we would obtain a point $z$ different from the $t_n$.
\end{proof}

If  $X$ consists only of  isolated points and $A$ is dense, then $A=X$. The theorem  also implies that a 
nonempty complete metric space with no isolated points must be 
uncountable (consider a   function constant on $X\setminus \{p\}$ and with a different value at $p$).  It follows that positive-radius closed balls in $X$, as they are complete, must be uncountable. Therefore, nonempty open balls in $X$ are uncountable.

In the context of Theorem~\ref{T1} we mention an interesting result that contrasts the difference between continuity on
a subset and continuity of the function restricted to a subset. This   is a consequence of theorems of
Blumberg and Sierpinsky--Zygmund (recent new proofs of these two theorems can be found in \cite{CMS}).
Every function $h:\mathbb R\to\mathbb R$ has a dense set $D\subset\mathbb R$ so that the restriction of $h$ to
$D$ is continuous (Blumberg), but there exists one such function $h$ such that   for every  subset $S\subset\mathbb R$ that has the cardinality of $\mathbb R$, the restriction of $h$ to $S$ is discontinuous (Sierpinsky--Zygmund).

\section  {The Examples.}\label{examples}  Let $X$ be a    metric space without isolated points, and $A$ a countable  subset of $X$ such that $A^c$ is dense.  For example, if    $X$ is a   complete metric space without isolated points and $A$ is any countable subset, its complement is dense since nonempty open balls in $X$ are uncountable.  We show that  there is  a function on $X$ that
is discontinuous precisely  on $A$. (As remarked earlier, this is a special case of a theorem of Young, Lebesgue, and Hahn.) In fact,  write $A=\{q_n:n\in\mathbb N\}$. Define a function $g$ on $A$ by setting $g(q_n)=1+1/2^n$, and 
when $x$ is not in $A$ let  $g(x)=1$. To show that $g$ is continuous outside $A$ let $x\notin A$ and $\varepsilon>0$.  Choose $n$ so that $1/2^n<\varepsilon$. 
Then choose $\delta>0$ so that the ball $B(x,\delta)$ avoids the finitely many points $q_1,q_2,\ldots,q_{n-1}$. Now let  $y$ be a point in $ B(x,\delta)$. If $y\notin A$, then $g(y)=1=g(x)$, so $|g(x)-g(y)|<\varepsilon$. If $y\in A$, then $g(y)=1+1/2^m$ for some  $m\geq n$. So 
 $|g(x)-g(y)|= 1/2^m<\varepsilon$. Finally, to see that $g$ is not continuous on $A$ we use that 
 the complement of $A$ is dense in $X$.  For each $q_n$ let $0<\varepsilon<1/2^n$. We can choose $y$ in $A^c$ that is arbitrarily close to $q_n$ and
 such that $|g(q_n)-g(y)|=1/2^n>\varepsilon$. Then $g$ is discontinuous at $q_n$.  
 
 The example also shows that one cannot hope to improve the theorem, as there are functions that are continuous on an uncountable set but not the whole space. 
 
 If we  restrict the domain of the function $g$ to the case when  $X=\mathbb R$,  then one can verify that $g$ is not differentiable at points in $A^c$  (and of course also at points in $A$). This is also holds for Thomae's function; for modified Thomae's functions that are differentiable on a countable set of irrationals see \cite{BRS}.
 
 We end by  noting  that the example above can also be modified to prove an equivalent form of the Young-Lebesgue-Hahn  theorem, which also has a proof in Gauld--Piotrowsk \cite{GP93}, Kim \cite{K99}: \textit{Let $X$ be a  metric space without isolated points. If $A$ is a countable union of closed sets, we show that  there is a function $g(x)$ which is discontinuous exactly on $A$.} 
 Before construction the function $g$ we observe that by taking complements we also obtain that given any set $G$ that is a countable intersection of open sets we have that $G=C(g)$. So countable intersections of open sets are sets of continuity points. This prompts one to ask if Theorem~\ref{T1} already proves Baire's theorem, in other words, if for metric spaces the Volterra property implies the Baire property. This has been shown to be the case in \cite{GL00}.
 
 We will need the fact that any metric space  $X$ without isolated points has a dense subset $D$ whose complement is also dense in $X$; this is clearly true for $\mathbb {R}^n$ and a proof of the general case can be found in   Kim \cite{K99}. We outline Kim's proof for completeness. For each $n\in\mathbb N$ let $S_{1/n}$ denote a subset of $X$ such that any two points in $S_{1/n}$ are at a distance greater than or equal to $1/n$ and that is maximal with respect to this property. Zorn's lemma implies that these sets exist. Now start with $S_1$. Then $X\setminus S_1$ must be nonempty and is a metric space without isolated points, so here we can consider $S_{1/2}$. In this way we obtain a pairwise disjoint sequence 
 $S_1\cup S_{1/2}\cup S_{1/3}\cdots$. Then the sets $D_1=\bigcup_n S_{1/2n}$ and $D_2=\bigcup_n S_{1/2(n-1)}$ are disjoint and dense.
 
 Write now  $A= \bigcup_{n\in\mathbb N} Q_n$, where each $Q_n$ is a closed subset of $X$, and we may assume they are increasing: $Q_n\subset Q_{n+1}$.  We modify  the function $g$ above  in the following way. First we define $n(x)$ to be the smallest positive integer such that $x\in Q_{n(x)}$. Define 
 $$ g(x) =
\begin{cases}
1, &\text{ if } x\notin A;\\
1+1/2^n, &\text{ if } x\in A, n=n(x),  \text{ and } x\in D;\\
1-1/2^n, &\text{ if } x\in A, n=n(x), \text{ and } x\notin D.\\
\end{cases}$$
 The proof that the modified function g is discontinuous only on A is similar to
the proof given in the example above.
 
 
 To show that $g$ is continuous outside $A$ let $x\notin A$ and $\varepsilon>0$. Choose $n$ so that $1/2^n<\varepsilon$. Then choose $\delta>0$ so that the ball $B(x,\delta)$ avoids the finitely many closed sets $Q_1, Q_2, \ldots, Q_{n-1}$. We can choose such $\delta$ because the finite union of closed sets $Q_1 \cup Q_2, \cup \ldots \cup Q_{n-1}$ is closed and thus not dense. Now let $y$ be a point in $B(x,\delta)$. If $y\notin A$, then $g(y)=1=g(x)$, so $|g(x)-g(y)|<\varepsilon$. If $y\in A$, then $g(y)=1\pm 1/2^m$ for some $m\geq n$. So $|g(x)-g(y)|= 1/2^m<\varepsilon$.
 
 Finally, to see that $g$ is not continuous on $A$, we note that both $D$ and $D^c$ are dense in $X$, and $X = D \sqcup D^c$. For any $x\in A$ we define $n=n(x)$  and let $0<\varepsilon<1/2^n$. Choose any $\delta$. Since $A = (D\cap A) \sqcup (D^c\cap A)$, either $x\in D$ or $x\in D^c$. If $x\in D$, since $D^c$ is dense in $X$, we can choose any $y\in B(x,\delta)$ such that $y\in D^c$, then $|g(x)-g(y)|\geq 1/2^n>\varepsilon$. Similarly, if $x\in D^c$, we choose any $y\in B(x,\delta)$ such that $y\in D$, then $|g(x)-g(y)|\geq 1/2^n>\varepsilon$. Thus $g$ is discontinuous at any $x \in A$.

 \bigskip

\noindent {\bf Acknowledgments.}
The  authors would like to thank  Frank Morgan for comments and suggestions. This started as a project by the second-named
author in a real analysis class taught by the first-named author at Williams. C.S. acknowledges support from NSF REU grants.  


\subsection*{Cesar E. Silva}
  Department of Mathematics and Statistics\\
     Williams College \\ Williamstown, MA 01267, USA. 
     csilva@williams.edu

\subsection*{Yuxin Wu}

     Management Science \& Engineering\\
Huang Engineering Center\\
Stanford University\\
475 Via Ortega\\
Stanford, CA 94305, USA.
 {yuxinwu@stanford.edu}


\end{document}